\documentclass[11pt,fleqn]{article}
\usepackage{amssymb}
\usepackage{amsfonts}
\usepackage{amsmath}
\usepackage{amsthm}
\usepackage{graphicx}
\usepackage{epsfig}
\usepackage{psfrag}
\usepackage{color}
\usepackage{dsfont}
\makeatletter
\xdef\@endgadget#1{{\unskip\nobreak\hfil\penalty50\hskip1em\hbox{}\nobreak
    \hfil#1\parfillskip=0pt\finalhyphendemerits=0\par}}
\def\@qedsymbol{${}_\blacksquare$}
\def\qed{\@endgadget{\@qedsymbol}}
\newtheorem{lemma}{Lemma}[section]

\newtheorem{example}[lemma]{Example}
\newtheorem{definition}[lemma]{Definition}

\newtheorem{proposition}[lemma]{Proposition}
\newtheorem{remark}[lemma]{Remark}
\newcommand{\mR}{\mathbb{R}}

\newcommand{\X}{\mathcal{X}}
\newcommand{\Z}{\mathcal{Z}}

\newcommand{\M}{\mathcal{M}}

\newcommand{\V}{\mathcal{V}}
\newcommand{\W}{\mathcal{W}}
\newcommand{\Y}{\mathcal{Y}}
\newcommand{\U}{\mathcal{U}}

\newcommand{\cL}{\mathcal{L}}

\newcommand{\pperp}{\perp \!\!\!\perp}
\newcommand{\bq}{\begin{equation}}
\newcommand{\eq}{\end{equation}}
\newcommand{\bma}{\begin{bmatrix}}
\newcommand{\ema}{\end{bmatrix}}
\newcommand{\Sin}{\mathrm{Sin\,}}

\def\BibTeX{{\rm B\kern-.05em{\sc i\kern-.025em b}\kern-.08em
    T\kern-.1667em\lower.7ex\hbox{E}\kern-.125emX}}

\title{\LARGE \bf Reciprocity of nonlinear systems}

\author{Arjan van der Schaft
\thanks{A.J. van der Schaft is with the Bernoulli Institute for Mathematics, Computer
Science and AI, Jan C. Willems Center for Systems and Control, University of Groningen, PO Box 407, 9700 AK, the
Netherlands
        {\tt\small A.J.van.der.Schaft@rug.nl}}
}

\begin{document}

\maketitle
\thispagestyle{empty}
\pagestyle{empty}

\section{Introduction}
One of the key contributions of the seminal paper \cite{willems72b} was the analysis of \emph{symmetry} (also called \emph{reciprocity}) of input-state-output systems, both from an external (input-output) and internal (state) point of view. The developed theory also included the combination of reciprocity with passivity, and the consideration of relaxation systems, which are passive reciprocal systems without any oscillatory behavior. The paper was motivated from a fundamental system-theoretic point of view (how is external structure reflected into internal structure), as well as by a wide range of application areas, including electrical network synthesis, thermodynamics, and viscoelastic materials. In fact, the use of the terminology 'reciprocal' goes back to Maxwell (Maxwell's reciprocal rule) and Onsager (the Onsager reciprocal relations). Recently, there is renewed interest in the study of reciprocal systems, and relaxations systems in particular, motivated by neuro-computing \cite{chaffey}, as well by obtaining simple control strategies for complex physical systems \cite{pates19,pates22}.

On the other hand, how powerful and elegant the results obtained in \cite{willems72b} may be, they only hold for \emph{linear} systems, and the extension to the nonlinear case, even for subclasses of nonlinear systems, is far from trivial. The present paper aims at taking some steps into this direction. This will be done by first recalling in Section \ref{sec:linear} some of the main features of the linear theory, and to reformulate them in a more geometric fashion, thus setting the stage for a proper nonlinear generalization.
Then, in Section \ref{sec:reciprocity}, the geometric definition of reciprocity is extended to general nonlinear state space systems; relying on the notion of a Lagrangian submanifold. Furthermore, it is shown how this definition is equivalent to the more classical definition of (pseudo-)gradient systems. This latter definition originates (for the case of systems without inputs and outputs) from the theory of dynamical systems, in particular the work \cite{BM1,BM2} on modeling and analysis of nonlinear RLC electrical networks. This was later extended to input-state-output systems in e.g. \cite{crouch, goncalves, vds83,cortes}. Furthermore, in this section special attention is paid to (pseudo-)gradient systems with a Hessian pseudo-Riemannian metric, and their relation to port-Hamiltonian systems (see already \cite{vds11,jeltsema}). Section \ref{sec:combination} deals with the combination of reciprocity and passivity; trying to generalize the elegant theory of \cite{willems72b}. The case of port-Hamiltonian systems with added reciprocity structure (see already \cite{vds11}) will turn out to be instrumental in reaching such a (partial) extension. Finally, Section \ref{sec:examples} provides a number of examples, while the conclusions and open problems are indicated in Section \ref{sec:conclusions}. 

The paper ends with two appendices about necessary background; Appendix A about some useful properties of Legendre transformation, and Appendix B about connections on manifolds.

\medskip

{\bf Notation:} The vector of partial derivatives of a function $S:\X \to \mR$ will be denoted either by $\frac{\partial S}{\partial x}(x)$ (column vector) or by $\nabla S(x)$ (often as a row vector). The Hessian matrix of $S$ will be denoted by $\nabla^2 S(x)$.

\medskip

{\bf Acknowledgements:} The writing of the paper profited from discussions with many people, including Rodolphe Sepulchre, Henk van Waarde, Tom Chaffey, Kanat Camlibel, Juan Peypouquet and Hans Schumacher.

\section{The linear case; recall and geometric interpretation}\label{sec:linear}
In this section I will briefly summarize the theory of reciprocal linear systems, and reciprocity combined with passivity, as originating from \cite{willems72b}. Furthermore, in order to set the stage for the extension to the nonlinear case, these results will be interpreted from a \emph{geometric} perspective; providing additional insights into the linear case as well.

\subsection{Reciprocity in the linear case}
We will start by recalling the geometric notion of a \emph{Lagrangian subspace}, generalizing the graph of a symmetric linear map. Let $\V$ be a finite-dimensional linear space with $\dim \V=n$. Consider the duality product $<w \mid v> , v \in \V, w\in \V^*$. In any basis for $\V$ and dual basis for $\V^*$ the duality product reduces to the vector product $w^\top v$, where we identify $v$ and $w$ with their coordinate vectors. Define on $\V \times \V^*$ the \emph{symplectic form}
\bq
\label{symplectic}
\langle (v_1,w_1), (v_2,w_2) \rangle := <w_1 \mid v_2> - <w_2  \mid  v_1>, \quad (v_1,w_1), (v_2,w_2) \in \V \times \V^*
\eq
After choosing linear coordinates for $\V$ and dual coordinates for $\V^*$ the symplectic form corresponds to the matrix $\bma 0 & -I_n \\ I_n & 0 \ema$.
\begin{definition}
A subspace $\cL \subset \V \times \V^*$ is Lagrangian if the symplectic form $\langle \cdot, \cdot \rangle$ is zero on $\cL$, and $\cL$ is maximal with respect to this property (i.e., there does not exist a subspace $\cL'$ with $\langle \cdot, \cdot \rangle$ zero on $\cL'$ and $\cL\subsetneqq \cL'$). 
\end{definition}
\begin{proposition}
A subspace $\cL \subset \V \times \V^*$ is Lagrangian if and only if $\cL=\cL^{\pperp}$, where ${}^{\pperp}$ denotes orthogonal companion with respect to $\langle \cdot, \cdot \rangle$. Any Lagrangian subspace satisfies $\dim \cL =\dim \V$.
\end{proposition}
An example of a Lagrangian subspace $\cL$ is the graph of a symmetric map $\V \to \V^*$, or $\V^* \to \V$.
In general, we have the following representation of Lagrangian subspaces.
\begin{proposition}
Consider a Lagrangian subspace $\cL \subset \V \times \V^*$. Take any set of coordinates for $\V$ and dual coordinates for $\V^*$. Then, possibly after permutation, we can split the coordinate vectors $v \in \V$ and correspondingly $w \in \V^*$, as
$v= \bma v_1 \\ v_2 \ema, \quad w= \bma w_1 \\ w_2 \ema$, with $\dim v_1=\dim w_1=k$, $\dim v_2=\dim w_2= n-k$,
such that $\cL$ is parametrized by $v_1,w_2$, and there exists a symmetric matrix $L$, correspondingly split as $L= \bma L_{11} & L_{12} \\[2mm] L_{21} & L_{22} \ema$, such that
\bq
 \label{eq:gen}
\cL=\{ (v_1,v_2,w_1,w_2) \mid \bma w_1 \\[2mm] v_2 \ema = \bma L_{11} & L_{12} \\[2mm] - L_{21} & - L_{22} \ema
 \bma v_1 \\[2mm] w_2 \ema \}
 \eq
\end{proposition}
 Note that symmetry of $L$ is equivalent to
 \bq
 \label{eq:gen1}
 \bma I_k & 0 \\[2mm] 0 & -I_{n-k} \ema \bma L_{11} & L_{12} \\[2mm] -L_{21} & -L_{22} \ema =
 \bma L_{11} & L_{12} \\[2mm] -L_{21} & -L_{22} \ema^\top \bma I_k & 0 \\[2mm] 0 & -I_{n-k} \ema
 \eq
Here any diagonal matrix $\sigma := \bma I_k & 0 \\ 0 & -I_{n-k} \ema$ with $k$ diagonal elements $1$ and $n-k$ diagonal elements $-1$ is called a {\it signature matrix}. 

Consider now a linear system on an $n$-dimensional linear state space $\X$
\bq
\label{linear}
\Sigma: \quad 
\begin{array}{rcl}
\dot{x} & = & Ax + Bu, \quad x \in \X, u\in \U \\[2mm]
y & = & Cx +Du, \quad y \in \Y,
\end{array}
\eq
where the $m$-dimensional linear input and output space $\U$ and $\Y$ are such that $\U = \Y^*$. Associate with the system $\Sigma$ its \emph{system map} 
\bq
\Sigma_m := \bma A & B \\[2mm] C & D \ema : \X \times \U \to \X \times \Y
\eq
Consider an invertible symmetric $n \times n$ matrix $G$, defining a (possibly indefinite) inner product on $\X$, as well as a linear map (denoted by the same letter) $G: \X \to \X^*$. Furthermore, consider an $m \times m$ signature matrix
$\sigma$, which can be regarded as the matrix representation of a linear map $\sigma: \Y \to \Y$ satisfying $\sigma^2=I$.
\begin{definition}
The linear system $\Sigma$ is called \emph{reciprocal} (with respect to $G$ and $\sigma$) if the system map $\Sigma_m: \X \times \U \to \X \times \Y$, composed with the map $\bma G & 0 \\ 0 & \sigma \ema : \X \times \Y \to \X^* \times \Y$, is \emph{symmetric}, that is
\bq
\label{reclinear}
\bma GA & GB \\[2mm] \sigma C & \sigma D \ema = \bma GA & GB \\[2mm] \sigma C & \sigma D \ema ^\top
\eq
Equivalently, the \emph{graph} of the map $\bma GA & GB \\ \sigma C & \sigma D \ema : \X \times \U \to \X^* \times \Y$ is a Lagrangian subspace of $\X \times \U \times \X^* \times \Y$.
\end{definition}
In view of \eqref{reclinear} any reciprocal system can be rewritten into the form
\bq
\label{pseudolinear}
\begin{array}{rcl}
G\dot{x} &= &- Px + C^\top \sigma u \\[2mm]
 y & = & Cx + D u, \quad \qquad \sigma D= D^\top \sigma,
\end{array}
\eq
where $P:=-GA$ satisfies $P=P^\top$. This is called a \emph{pseudo-gradient system} (and a \emph{gradient system} if $G>0$). The quadratic function $\frac{1}{2}x^\top P x$ is called the internal \emph{potential function} of the pseudo-gradient system.

As noted in \cite{willems72b} the matrix $G$ can be obtained from input-output data. In fact, along any solution of $\Sigma$ one computes
\bq
\begin{array}{l}
\frac{d}{dt} x^\top (t)Gx(-t) = \\[2mm]
\left[x^\top(t)A^\top + u^\top(t)B^\top \right]Gx(-t) - x^\top(t)G \left[Ax(-t) + Bu(-t) \right],
\end{array}
\eq
and thus by \eqref{reclinear}
\bq
\frac{d}{dt} x^\top (t)Gx(-t) = u^\top (t) \sigma y(-t) - (\sigma y)^\top(t) u(-t)
\eq
Hence, by considering solutions with $u(t)=0$ for $t\geq 0$, and integrating over $[0,\infty)$, one obtains
$x^\top (\infty)Gx(-\infty) - x^\top(0)Gx(0) = - \int_0^\infty (\sigma y)^\top (t) u(-t) dt$. Thus if either $x(-\infty)$ or $x(\infty)$ equals zero, then $G$ is determined by
\bq
\label{Ghankel}
x^\top(0)Gx(0) = \int_0^\infty (\sigma y)^\top (t)  u(-t) dt
\eq
for any $x(0) \in \X$.
\begin{remark}
{\rm As recently noted and explored in \cite{chaffey}, see also \cite{willems72b}, the right-hand side of \eqref{Ghankel} is equal to the quadratic form defined by the symmetric \emph{Hankel operator} of the reciprocal system $\Sigma$. (Recall that the Hankel operator maps input functions $u(-t), t \in [0,\infty)$, to output functions $y(t), t \in [0,\infty)$.)}
\end{remark}
Alternatively, see \cite{willems72b}, in case of a controllable and observable system $\Sigma$, the matrix $G$ is uniquely determined as the \emph{state space isomorphism} $G: \X \to \X^*$ between $\Sigma$ and its \emph{dual system} (with respect to $\sigma$), defined as
\bq
\Sigma^d: \quad 
\begin{array}{rcl}
\dot{p} & = & A^Tp + C^\top \sigma u^d, \quad  p \in \X^*, u^d \in \U^* \\[2mm]
y^d & = &  B^\top  p +  D^\top \sigma u^d, \quad y^d \in \Y^*
\end{array}
\eq
with dual inputs $u^d$ and outputs $y^d$. Indeed, reciprocity means that the input-output behavior of the dual system $\Sigma^d$ is equal to that of $\Sigma$. This is the same as that the impulse response matrix $W(t)= Ce^{At}B \, + \, D\delta(t)$ of $\Sigma$ is \emph{symmetric} with respect to $\sigma$, i.e., $\sigma W(t)= W^\top (t) \sigma$. Equivalently, the transfer matrix $H(s)=C(Is - A)^{-1}B + D$ satisfies $\sigma H(s)= H^\top (s) \sigma$. Notice that this implies that any single-input single-output linear system is reciprocal.

\subsection{Passivity of linear systems}
A linear system $\Sigma$ is \emph{passive} if there exists an $n \times n$ symmetric matrix $Q\geq 0$ satisfying the dissipation inequality $\frac{d}{dt} \frac{1}{2} x^\top Q x \leq y^\top u$ for any $x,u$, or equivalently the LMI
\bq
\label{paslinear}
\bma Q & 0 \\[2mm] 0 & I \ema \bma -A & -B \\[2mm] C & D \ema +
\bma -A^\top & C^\top \\[2mm] -B^\top & D^\top \ema  \bma Q & 0 \\[2mm] 0 & I \ema \geq 0
\eq
The quadratic function $\frac{1}{2} x^\top Q x \geq 0$ is called a \emph{storage function} for the system $\Sigma$. If the requirement $Q\geq 0$ is dropped then the system is \emph{cyclo-passive}. The following result is known from the literature; the short proof given below is an adaptation of the one given in \cite{cam14}.
\begin{proposition}
Suppose $Q=Q^\top$ is a solution to \eqref{paslinear}. Then $\ker Q$ is $A$-invariant and contained in $\ker C$. In particular, if the system is observable then necessarily $\ker Q=0$.
\end{proposition}
\begin{proof}
Denote the symmetric matrix on the left-hand side of the LMI \eqref{paslinear} by $\Pi$.
Let $z \in \ker Q$. Then $\bma z^\top & 0  \ema \Pi \bma z \\ 0  \ema=0$, and thus, since $\Pi=\Pi^\top \geq 0$, also $\Pi \bma z \\ 0  \ema=0$. This means $Az \in \ker Q$ and $Cz=0$.
\end{proof}
The LMI \eqref{paslinear} admits the following interpretation, using the geometric notion of a (maximally) \emph{monotone subspace}, which is defined as follows.
Consider again an $n$-dimensional linear space $\V$. Next to the symplectic form $\langle \cdot, \cdot \rangle$, there is another canonical bilinear form on $\V \times \V^*$ (where the minus sign in \eqref{symplectic} is replaced by a plus sign), namely
\bq
\langle \langle(v_1,w_1), (v_2,w_2) \rangle \rangle := <w_1 \mid v_2> + <w_2  \mid  v_1>, \; (v_1,w_1), (v_2,w_2) \in \V \times \V^*
\eq
In linear coordinates for $\V$ and dual coordinates for $\V^*$ this corresponds to the matrix $\bma 0 & I_n \\ I_n & 0 \ema$, having $n$ singular values $+1$ and $n$ singular values $-1$.
\begin{definition}
A subspace $\M \subset \V \times \V^*$ is \emph{monotone} if the quadratic form defined by $\langle \langle  \cdot, \cdot \rangle \rangle$ is nonnegative on $\M$, and \emph{maximally monotone} if $\M$ is maximal with respect to this property (that is, there does not exist a subspace $\M'$ with $\langle \langle \cdot, \cdot \rangle \rangle$ nonnegative on $\M'$ and $\M\subsetneqq \M'$).
\end{definition}
We have the following results (see \cite{vds23} for a proof).
\begin{proposition} 
\label{monotoneprop}Any monotone subspace $\M \subset \V \times \V^*$ satisfies $\dim \M \leq \dim \V$. A monotone subspace is maximally monotone iff $\dim \M = \dim \V$. Any monotone subspace $\M \subset \V \times \V^*$ can be represented as $\M=\{ (v,w) \mid v= M_1 \lambda, w=M_2 \lambda, \lambda \in \Z$ for some maps $M_1: \Z \to \V, M_2: \Z \to \V^*$ such that the map $M:=M_2^*M_1: \Z \to \Z^*$ is monotone, i.e., $<z, Mz> \geq 0$ for all $z \in \Z$. Conversely, for any such $M_1,M_2$ the subspace $\M$ is monotone.
\end{proposition}
Now associate to the LMI \eqref{paslinear} the subspace $\M \subset \X \times \Y \times \X^* \times \U$ given as\footnote{The minus sign in front of $\dot{x}$ can be explained by considering the port-Hamiltonian formulation of passive systems; see e.g. \cite{jeltsema}.}
\bq
\label{monotone}
\M=\{(-\dot{x},y,z,u) \mid \dot{x} = Ax + Bu, y=Cx + Du, z=Qx \},
\eq
Thus $\M$ is equal to the image of the map
\bq
\label{monmap}
\bma M_1 \\[2mm] M_2 \ema := \bma - A & - B \\ C & D \\[-2mm] --&--\\[-1mm]Q & 0 \\ 0 & I \ema
\eq
Hence by Proposition \ref{monotoneprop}, \eqref{paslinear} is equivalent to $\M$ being a \emph{monotone subspace}. As noted above, if the system is \emph{observable} then $\ker Q=0$. The structure of the map \eqref{monmap} then furthermore implies
\begin{proposition}
If the passive system \eqref{linear} is observable then $\dim \M = \dim \X + \dim \Y$ and $\M$ is \emph{maximally} monotone.
\end{proposition}
\begin{remark}
{\rm Furthermore, if $\Sigma$ is not observable, but $Q$ is such that $\ker Q \subset \ker A \cap \ker C$, then $\M$ can be \emph{extended} to a maximally monotone subspace.}
\end{remark}
Since $M_2^\top M_1= \bma Q & 0 \\ 0 & I \ema \bma - A & - B \\ C & D \ema$ is a monotone map, the same holds for  $\bma - A & - B \\ C & D \ema \bma Q^{-1} & 0 \\ 0 & I \ema$. By factorizing this latter monotone matrix into its skew-symmetric and positive-definite symmetric part, one immediately obtains a \emph{port-Hamiltonian formulation}; see Section \ref{subsec:ph}.

%

\subsection{Combining reciprocity and passivity for linear systems}
\label{subsec:recpas}
One of the fundamental contributions of \cite{willems72b} is to show how reciprocity can be \emph{combined} with passivity in order to derive a state space realization exhibiting a particularly insightful and useful form. 

In general a passive system admits many $Q \geq 0$ satisfying the LMI \eqref{paslinear}. The key observation is that if $Q>0$ is a solution to \eqref{paslinear}, and additionally the system is reciprocal with respect to $G$ and $\sigma$, then by combining \eqref{paslinear} and \eqref{reclinear} also the positive definite matrix $Q' := GQ^{-1}G$ is a solution of \eqref{paslinear}. By an application of Brouwer's fixed point theorem \cite{willems76} (or more mundane methods, see \cite{willems72b}) it then follows that there exists a $Q>0$ satisfying \eqref{paslinear} \emph{and}
\bq
\label{complinear}
Q = GQ^{-1}G
\eq
A solution $Q$ to \eqref{paslinear} satisfying \eqref{complinear} is said to be \emph{compatible} with $G$. 
Compatible $Q$ define storage functions with a clear physical relevance. 
For example, in electrical network synthesis the expression $\frac{1}{2}x^\top Q x$ for a compatible $Q$ corresponds to the energy stored at the capacitors and inductors of an RLCT synthesis \emph{without} gyrators of a symmetric positive real transfer matrix \cite{willems72b}.

In general, there may still exist many $Q$ satisfying both \eqref{paslinear} and the compatibility condition \eqref{complinear}. On the other hand, as shown in \cite{vds11,willems72b}, if $G>0$ then the only compatible $Q$ is $Q=G$. A passive and reciprocal system with $G>0$ was called\footnote{In fact, in \cite{willems72b} a different, but equivalent, definition in terms of the impulse response matrix was given; see also \cite{chaffey}.} a \emph{relaxation system} in \cite{willems72b}. It follows that the unique compatible storage function of a relaxation system is given as $\frac{1}{2} x^\top G x$. In particular, this unique compatible storage function is determined by the input-output behavior of the system; cf. \eqref{Ghankel}.

\subsection{Port-Hamiltonian formulation of linear passive pseudo-gradient systems}
\label{subsec:ph}
In this subsection we recall, largely based on \cite{vds11}, the (physically) insightful \emph{port-Hamiltonian} formulation of passive pseudo-gradient systems.

Consider throughout an observable pseudo-gradient system \eqref{pseudolinear}, for simplicity of exposition with $\sigma=I$.
Now suppose the system is also {\it passive}. Then, as discussed above, there exists a compatible $Q>0$. As proved in \cite{vds11} this means there exist coordinates  $x=(x_1,x_2)$, with $\dim x_1=k, \dim x_2=n-k$, for $\X$ such that 
\begin{equation}\label{compatible1}
Q = \begin{bmatrix} Q_1 & 0 \\ 0 & Q_2 \end{bmatrix} \, , \quad G = \begin{bmatrix} Q_1 & 0 \\ 0 & - Q_2 \end{bmatrix}
\end{equation}
in which the pseudo-gradient system takes the correspondingly split form
\begin{equation}\label{gradient2}
\begin{array}{rcl}
\begin{bmatrix} Q_1 & 0 \\ 0 & -Q_2 \end{bmatrix} \dot{x} & = & - \begin{bmatrix} P_1 & P_c \\ P_c^\top & P_2 \end{bmatrix}x +  \begin{bmatrix} C^\top_1 \\ 0 \end{bmatrix} u, \quad P_1 \geq 0, \, P_2 \leq 0 \\[4mm]
y & = & \begin{bmatrix} C_1 & 0 \end{bmatrix} x + Du, \qquad D=D^\top
\end{array}
\end{equation}
In fact, \eqref{gradient2} follows by considering $Q,G$ as in \eqref{compatible1}. From $C= B^\top G = B^\top Q$ one obtains $C= \begin{bmatrix} C_1 & 0 \end{bmatrix}$, with $C_1$ an $m \times k$ matrix. 
Substitution of $A= - G^{-1}P$ into $A^\top Q + QA  \leq 0$ then leads to 
$PG^{-1}Q + QG^{-1}P \geq 0$, which in view of (\ref{compatible1}) yields  $P_1 \geq 0, P_2 \leq 0$.

Multiplying now the second part of the differential equations (corresponding to $\dot{x}_2$) on both sides by a minus sign, one obtains the equivalent representation
\begin{equation}\label{gradient3}
\begin{array}{rcl}
\begin{bmatrix} Q_1 & 0 \\ 0 & Q_2 \end{bmatrix} \dot{x} & = & - 
\begin{bmatrix} P_1 & P_c \\ -P_c^\top & -P_2 \end{bmatrix}x +  \begin{bmatrix} C^\top_1 \\ 0 \end{bmatrix} u \\[4mm]
y & = & \begin{bmatrix} C_1 & 0 \end{bmatrix} x + Du
\end{array} 
\end{equation}
Finally, defining the \emph{new} state vector $z = (z_1, z_2)$ by $z_1=Q_1x_1, z_2=Q_2x_2$, the system takes the port-Hamiltonian form
\begin{equation}\label{portham1}
\begin{array}{rcl}
\dot{z} & = & \left( \begin{bmatrix} 0 & - P_c \\ P_c^\top & 0 \end{bmatrix} - \begin{bmatrix} P_1 & 0 \\ 0 & -P_2 \end{bmatrix} \right) \begin{bmatrix} Q_1^{-1} & 0 \\ 0 & Q_2^{-1}\end{bmatrix}z + 
\begin{bmatrix} C^\top_1 \\ 0 \end{bmatrix} u \\[6mm]
y & = & \begin{bmatrix} C_1 & 0 \end{bmatrix} \begin{bmatrix} Q_1^{-1} & 0 \\ 0 & Q_2^{-1}\end{bmatrix}z + Du,
\end{array} 
\end{equation}
with Hamiltonian (stored energy) $\frac{1}{2} z_1^\top Q_1^{-1}z_1 + \frac{1}{2}z_2^\top Q_2^{-1}z_2$.
We conclude that the matrix $P_c$ underlies the skew-symmetric interconnection matrix of the port-Hamiltonian system, corresponding to a \emph{lossless coupling} between the two types of energy storage $\frac{1}{2} z_1^\top Q_1^{-1}z_1$ and $\frac{1}{2}z_2^\top Q_1^{-1}z_2$. On the other hand, the symmetric matrices $P_1\geq 0$ and $-P_2 \geq 0$ correspond to \emph{energy dissipation}. 
In the Brayton-Moser formulation of RLC networks (see Example \ref{ex:BM} for the original nonlinear version) the internal potential function was therefore called the \emph{mixed potential}. In the case of a \emph{relaxation system} $Q=G$, implying that there is only \emph{one} type of energy-storage \cite{willems72b}, and that the matrices $P_c$ and $P_2$ are absent.

\section{Reciprocity of nonlinear systems}\label{sec:reciprocity}
The definition of reciprocity generalizes to the nonlinear case as follows. Instead of the pure symmetry condition of the linear case (the influence of the $i$-the input component on the $j$-th output component is the same as the influence of the $j$-th input on the $j$-th output), we consider some sort of symmetry of the influence of the \emph{variational} inputs on the \emph{variational} outputs. In a mechanical system context this is also referred to as \emph{Maxwell's reciprocal rule} (or theorem).

First, recall the definition of a \emph{Lagrangian submanifold}, generalizing the notion of a Lagrangian subspace as discussed in the previous section. Consider a (smooth) manifold $\M$ and its cotangent bundle $T^*\M$.  Any cotangent bundle $T^*\M$ is endowed with a natural symplectic form $\omega$. Taking local coordinates $x_1, \cdots, x_n$ for $\M$, and the corresponding natural coordinates $(x,p)=x_1, \cdots, x_n, p_1, \cdots,p_n$ for $T^*\M$, this symplectic form is given as $\omega = \sum_{i=1}^n dp_i \wedge dx_i =: dp \wedge dx$. 
\begin{definition}
A Lagrangian submanifold $\cL \subset T^*\M$ is a submanifold restricted to which $\omega$ is zero and which is maximal with respect to this property, i.e., there does not exist $\cL'$ restricted to which $\omega$ is zero and $\cL\subsetneqq \cL'$. 
\end{definition}
It follows that the dimension of any Lagrangian submanifold $\cL \subset T^*\M$ equals $n$ ($= \dim \M$). If $\cL \subset T^*\M$ can be locally parametrized by the coordinates $x_1, \cdots, x_n$, then there exists locally a function $V:\M \to \mR$ such that
\bq
\label{lagrangian}
\cL= \{ (x,p) \mid p_i= \frac{\partial V}{\partial x_i}(x), i=1, \cdots,n \},
\eq
and, conversely, for any $V$ the submanifold $\cL$ defined by \eqref{lagrangian} is a Lagrangian submanifold. (This extends the concept of symmetric maps in the linear case.)

More generally, given natural local coordinates $(x,p)$ for $T^*\M$, for any Lagrangian submanifold $\cL \subset T^*\M$ one can always split the index set $\{1, \cdots,n\}$ into two complementary subsets $I, J$, and find a function $V$ of the variables $(x_I,p_J)$, where $x_I$ denotes the vector of coordinates $x_i,i \in I$, and $p_J$ the vector of coordinates $p_j,j \in J$, such that locally \cite{arnold}
\bq
\label{lagrangian1}
\cL= \{ (x,p) \mid p_i= \frac{\partial V}{\partial x_i}(x_I,p_J), i \in I,  \; x_j= - \frac{\partial V}{\partial p_j}(x_I,p_J), j \in J\}
\eq
Conversely, any set \eqref{lagrangian1} for some $V$ is a Lagrangian submanifold. Extending the discussion of Lagrangian subspaces, the splitting $\{1, \cdots,n\}= I \cup J$ generalizes the use of signature matrices. (Note the minus sign in \eqref{lagrangian1}.)

Now consider a nonlinear system on an $n$-dimensional state space manifold $\X$
\bq
\label{system}
\Sigma: \quad 
\begin{array}{rcl}
\dot{x} & = & F(x,u), \quad x \in \X, u\in \U \\[2mm]
y & = & H(x,u), \quad y \in \Y
\end{array}
\eq
where, as before, the linear $m$-dimensional input and output spaces $\U$ and $\Y$ are such that $\U = \Y^*$.
Associate with the system $\Sigma$ its \emph{system map} $\Sigma_m$ given by
\bq
\Sigma_m : (x,u) \in \X \times \U \mapsto \left( F(x,u), H(x,u) \right) \in T_x\X \times \Y,
\eq
where $T_x\X$ denotes the tangent space to $\X$ at $x \in \X$.
\begin{remark}{\rm
Under smoothness assumptions it follows that the \emph{graph} of the system map $\Sigma_m$ is a submanifold of $T\X \times \U \times \Y$. Combining the input and output spaces $\U$ and $\Y$ into a space of external variables $\W$, and allowing the possibility of \emph{algebraic constraints} this suggests to define a general smooth nonlinear system with state space $\X$ as a general submanifold of $T\X \times \W$.}
\end{remark}
Consider an $m \times m$ signature matrix $\sigma : \Y \to \Y$. Furthermore, consider an invertible symmetric $n \times n$ matrix $G(x)$, smoothly depending on $x \in \X$. This defines a, possibly indefinite, Riemannian metric on $\X$, called a \emph{pseudo-Riemannian metric}. Extending the linear case, the pseudo-Riemannian metric $G$ on $\X$ defines a vector bundle map $G: T\X \to T^*\X$, for simplicity denoted by the same letter $G$, given by $v \in T_x\X \mapsto G(x)v \in T^*_x\X, x \in \X$.
\begin{definition}
The nonlinear system $\Sigma$ is called \emph{reciprocal} (with respect to the pseudo-Riemannian metric $G$ on $\X$ and signature matrix $\sigma$) if the graph of the system map $\Sigma_m$ composed with the map $(v,y) \in T_x\X \times \U \mapsto (G(x)v,\sigma y) \in T_x\X \times \Y $, that is the graph of the nonlinear map
\bq
(x,u) \in \X \times \U  \mapsto \left( G(x)F(x,u), \sigma H(x,u) \right) \in T_x^*\X \times \Y
\eq
is a Lagrangian submanifold of $T^*\X \times \U \times \Y$.
\end{definition}
\begin{remark}
{\rm 
Equivalently \cite{vds83}, consider the symplectic form $d G(x) \, \wedge \, dx + d \, \sigma \, y \wedge du$ on $T \X \times \U \times \Y$ (the pull-back of the standard symplectic form on $T^* \X \times \U \times \Y$ by the map $(v,y) \in T_x \X \times \U \mapsto (G(x)v,\sigma y)$ ). Then reciprocity is the same as the graph of the system map $\Sigma_m$ being a Lagrangian submanifold of $T \X \times \U \times \Y$. 
}
\end{remark}
\begin{proposition}
The nonlinear system $\Sigma$ given by \eqref{system} is reciprocal (with respect to $G$ and $\sigma$) if and only if $F(x,u)$ and $H(x,u)$ satisfy for all $x,u$
\bq
\label{rec}
\begin{array}{l}
\frac{\partial G(x) F(x,u)}{\partial x} = \left(\frac{\partial G(x)F(x,u)}{\partial x}\right)^\top \\[2mm]
\sigma \, \frac{\partial H(x,u)}{\partial u} = \left(\frac{\partial H(x,u)}{\partial u}\right)^\top \sigma \\[2mm]
G(x) \frac{\partial F(x,u)}{\partial u} = \left(\frac{\partial H(x,u)}{\partial x}\right)^\top \sigma
\end{array}
\eq

\end{proposition}
\begin{proof}
The symplectic form on $T^* \X \times \U \times \Y$ is zero on $\cL$ if for all two tangent vectors to $\cL$ at $(x,u) \in \X \times \U$, indexed by $1$ and $2$,
\bq
\bma 
\delta_2 x  \\[2mm] \delta_2 u \\[2mm]  \delta_2 [G(x)F(x,u)] \\[2mm] \delta_2 [\sigma H(x,u)] \ema^\top
\bma
0 & 0 & -I_n & 0 \\[2mm]
0 & 0 & 0 & -I_m \\[2mm]
I_n & 0 & 0 & 0 \\[2mm]
0 & I_m & 0 & 0
\ema
\bma
\delta_1x \\[2mm] \delta_1u \\[2mm]   \delta_1 [G(x)F(x,u)] \\[2mm]  \delta_1 [\sigma H(x,u)]
\ema
= 0
\eq
for any pair of tangent vectors $\delta_1 x, \delta_2 x, \delta_1 u, \delta_2$ to $\X \times \U$. Written out this immediately yields \eqref{rec}. Furthermore, by definition of $\cL$ we have $\dim \cL = \dim \X + \dim \U = n + m$, and thus $\cL$ is Lagrangian.
\end{proof}
An important subclass are \emph{affine} nonlinear systems, where
\bq
\label{systemaffine}
F(x,u)= f(x) + g(x)u, \; H(x,u)=h(x) + k(x)
\eq
for an $n \times m$ matrix $g(x)$ and $m \times m$ matrix $k(x)$, both depending smoothly on $x$. In this case the conditions for reciprocity are seen to specialize to
\bq
\label{rec1}
\begin{array}{l}
 \frac{\partial G(x)f(x)}{\partial x} = \left(\frac{\partial G(x)f(x)}{\partial x}\right)^\top \\[2mm]
\frac{\partial G(x)g_j(x)}{\partial x} = \left(\frac{\partial G(x)g_j(x)}{\partial x}\right)^\top , \quad j=1, \cdots, m\\[2mm]
\sigma k(x) = k^\top (x) \sigma \\[2mm]
G(x) g(x) = \left(\frac{\partial h(x)}{\partial x}\right)^\top \sigma 

\end{array}
\eq
where $g_j(x)$ is the $j$-th column of the matrix $g(x)$.

\subsection{Pseudo-gradient systems}
Nonlinear systems satisfying the reciprocity conditions \eqref{rec} are also known as \emph{pseudo-gradient systems}; see \cite{crouch, goncalves, cortes, vds83} and the references quoted therein. Indeed, the reciprocity conditions \eqref{rec} are precisely the \emph{integrability} conditions for the local existence of a potential function.
\begin{definition}
A nonlinear system $\Sigma$ is a \emph{pseudo-gradient system} (with respect to a pseudo-Riemannian metric $G$ and signature matrix $\sigma$) if there exists locally a function $V: \X \times \U \to \mathbb{R}$ such that the system is of the form
\bq
\label{gradient}
\begin{array}{rcl}
G(x) \dot{x} & = &  -\frac{\partial V}{\partial x}(x,u) \\[2mm]
\sigma y & = & - \frac{\partial V}{\partial u}(x,u) 
\end{array}
\eq
i.e., $G(x)F(x,u) =  - \frac{\partial V}{\partial x}(x,u), \sigma H(x,u) = -\frac{\partial V}{\partial u}(x,u)$. If $G(x)>0$ (and thus $G$ is a true Riemannian metric), then $\Sigma$ is a \emph{gradient system}. The function $V$ is called a \emph{potential function}.
\end{definition}

An important subclass of pseudo-gradient systems occurs if the pseudo-Riemannian metric $G$ is of a special type, namely a \emph{Hessian} one. See e.g. \cite{shim,duistermaat} for more background on Hessian Riemannian metrics. 
\begin{definition}
Let $\X$ be a connected convex domain in some $n$-dimensional linear space. The pseudo-Riemannian metric $G$ on $\X$ is Hessian if there exists a function $K: \X \to \mR$ such that the $(i,j)$-th element of the matrix $G(x)$ is given as
\begin{equation}\label{int}
G_{ij}(x) = \frac{\partial^2 K}{\partial x_i \partial x_j}(x), \quad i,j=1, \cdots,n
\end{equation}
\end{definition}
\begin{remark}
{\rm More generally, $\X$ can be taken to be a \emph{manifold} endowed with a \emph{flat connection}, where the Hessian is defined with respect to this connection on $\X$; see e.g. \cite{shim}. In fact, for a connection on $\X$ with Christoffel symbols $\Gamma_{ij}^k$ the Hessian of $K: \X \to \mR$ is defined as the $n \times n$ matrix with $(i,j)$-th element $\frac{\partial^2 K}{\partial x_i \partial x_j}(x) - \sum_{k=1}^n \Gamma_{ij}^k(x) \frac{\partial K}{\partial x_k}(x)$.}
\end{remark}
The Hessian pseudo-Riemannian metric determined by $K$ will be denoted by $\nabla^2 K(x)$. Note that by definition the Hessian $\nabla^2 K(x)$ is assumed to be \emph{invertible}.
The following characterization of Hessian pseudo-Riemannian metrics was provided in \cite{shim,vds11}.
\begin{proposition}
A pseudo-Riemannian metric $G$ on an $n$-dimensional convex connected domain $\X$ is Hessian if and only if the components $G_{ij}$ of the matrix $G$ satisfy
\begin{equation}\label{inthessian}
\frac{\partial G_{jk}}{\partial x_i}(x) = \frac{\partial G_{ik}}{\partial x_j}(x), \quad i,j,k=1, \cdots,n
\end{equation}
\end{proposition}
\begin{proof}
(Only if) If $G_{ij}(x) = \frac{\partial^2 K(x)}{\partial x_i \partial x_j}$ then obviously $\frac{\partial G_{jk}(x)}{\partial x_i} = \frac{\partial^3 K(x)}{\partial x_i \partial x_j \partial x_k}=\frac{\partial G_{ik}(x)}{\partial x_j}$.
(If) Condition \eqref{inthessian}, together with the fact that $\X$ is a convex connected domain, guarantees the existence of functions $\chi_k$ such that $G_{jk}(x) = \frac{\partial \chi_k}{\partial x_j}(x), j,k=1, \cdots,n$. Then by symmetry of $G(x)$
\begin{equation}
\label{hessian1}
\frac{\partial \chi_k}{\partial x_j}(x) = G_{jk}(x) = G_{kj}(x) = \frac{\partial \chi_j}{\partial x_k}(x), \quad j,k=1, \cdots,n,
\end{equation}
which is the integrability condition guaranteeing the existence of a function $K$ satisfying $\chi_j(x) = \frac{\partial K(x)}{\partial x_j}, j=1,\cdots,n$.
\end{proof}
Interestingly, the reciprocity conditions \eqref{rec} simplify as follows in case the pseudo-Riemannian metric is Hessian.
\begin{proposition}
\label{prop:monotone}
Consider an $n$-dimensional convex connected domain $\X$ endowed with a Hessian pseudo-Riemannian metric $G(x)=\nabla^2 K(x)$. Then the nonlinear system \eqref{system} is reciprocal (and thus a pseudo-gradient system \eqref{gradient}) if and only if the second and third line of \eqref{rec} are satisfied, and furthermore for all $x,u$
\bq
\label{rechessian}
G(x)\frac{\partial F(x,u)}{\partial x} = \left(\frac{\partial F(x,u)}{\partial x}\right)^\top G(x) 
\eq
\end{proposition}
\begin{proof}
The first line of \eqref{rec} can be written out as (omitting arguments)
\bq
\label{integrability}
\sum_\ell G_{k\ell} \frac{\partial }{\partial x_i}F_\ell + \sum_\ell \frac{\partial G_{k\ell}}{\partial x_i}F_\ell   = 
\sum_\ell G_{i\ell} \frac{\partial }{\partial x_k}F_\ell + \sum_\ell \frac{\partial G_{i\ell}}{\partial x_k}F_\ell 
\eq
for all $i,k,\ell=1, \cdots, n,$, with subscripts $i,k,\ell$ denoting the $i$-th, $k$-th, and $\ell$-th component. 
Since $G$ is Hessian, \eqref{inthessian} holds, implying that 
\bq
\sum_\ell \frac{\partial G_{k\ell}}{\partial x_i}(x)F_\ell (x,u) = \sum_\ell \frac{\partial G_{i\ell}}{\partial x_k}(x)F_\ell (x,u)
\eq
Thus the first line of \eqref{rec} reduces to \eqref{rechessian}.
\end{proof}
Hessian pseudo-gradient systems admit the following \emph{alternative} description, which turns out to be related to a \emph{port-Hamiltonian} formulation. Consider the pseudo-gradient system \eqref{gradient} with Hessian pseudo-Riemannian metric $G(x) = \nabla^2 K(x)$. Let $z=\nabla K(x)$ (the row vector of partial derivatives of $K$). By the assumptions on $K$ (namely $\nabla^2 K(x)$ invertible) the mapping $x \mapsto z=\nabla K(x)$ is an injective mapping, with well-defined inverse on its co-domain $\Z$ given by $z \mapsto x=\nabla K^*(z)$, where $K^*$ is the Legendre transform of $K$ (see Appendix A). Since $\dot{z}= \nabla^2 K(x) \dot{x} $, it follows that \eqref{gradient} can be rewritten as the system on $\Z$ given as
\bq
\label{gradienthessian}
\begin{array}{rcl}
\dot{z} & = &  -\frac{\partial V}{\partial x}(\nabla K^*(z),u) \\[2mm]
\sigma y & = & - \frac{\partial V}{\partial u}(\nabla K^*(z),u) 
\end{array}
\eq
By imposing additional \emph{convexity} conditions on the potential function $V$ these systems fall into the class of maximal (cyclically) monotone port-Hamiltonian systems as explored in \cite{cams22,cams13}.
\begin{proposition}
Assume that $V: \X \times \U \to \mR$ is a \emph{convex} function of $x,u$. Then \eqref{gradienthessian} for $\sigma=-I$ is a maximal cyclically monotone port-Hamiltonian system, with Hamiltonian $K^*(z)$. Alternatively, if $V(x,u)$ is convex in $x$ for constant $u$ and concave in $u$ for constant $x$, then \eqref{gradienthessian} for $\sigma=I$ is a maximal monotone port-Hamiltonian system.
\end{proposition}
We refer to \cite{cams22} for the proof; exploiting the fact that the graph of the derivative of a convex function is a maximal cyclically monotone relation. Monotone port-Hamiltonian systems enjoy the incremental passivity property
\bq
<\frac{\partial K^*}{\partial z}(z_1) - \frac{\partial K^*}{\partial z}(z_2),\dot{z}_1 - \dot{z}_2>\;  \leq \; <y_1 -y_2, u_1 -u_2>
\eq
for any two trajectories $(z_i,u_i,y_i), i= 1,2$, of \eqref{gradienthessian}. Note however that the Hamiltonian $K^*$ is \emph{not} necessarily bounded from below, in which case it does not qualify as a storage function.

As an example, let $V(x,u)= P(x) - x^\top gu$ for some matrix $g$ and convex internal potential function $P$. Then $V(x,u)$ is convex in $x$ for constant $u$ and concave in $u$ for constant $x$, leading to the maximal monotone port-Hamiltonian system
\bq
\begin{array}{rcl}
\dot{z} & = &  -\frac{\partial P}{\partial x}(\nabla K^*(z)) + gu \\[2mm]
y & = &  g^\top \nabla K^*(z) 
\end{array}
\eq

\subsection{External characterization of pseudo-gradient systems}
A main question concerns the \emph{input-output characterization} of nonlinear reciprocity.
Recall that in the linear case reciprocity is simply equivalent to symmetry of the impulse response matrix or of the transfer matrix. However (as already noted in \cite{willems72b}), this is \emph{not} easily extended to nonlinear systems, or even to time-varying linear systems. 

In \cite{cortes} an external characterization of pseudo-gradient systems has been given along the following lines.
Consider the class of \emph{affine} nonlinear systems \eqref{systemaffine} with $k=0$ and signature matrix $\sigma=I$, and consider along any solution $x(t),u(t),y(t)$ of the system the \emph{variational system} \cite{crouchvds}
\bq
\begin{array}{rcl}
\dot{\delta x}(t) & = & \frac{\partial f}{\partial x}(x(t)\delta x(t) + \sum_j u_j\frac{\partial g_j}{\partial x}(x(t))\delta x(t) + \sum_j g_j(x(t)) \delta u(t) \\[2mm]
\delta y(t) & = & \frac{\partial h}{\partial x}(x(t))\delta x(t)
\end{array}
\eq
where $\delta u, \delta x, \delta y$ are the \emph{variational inputs, states, and outputs}. The main theorem in \cite{cortes} shows that, under technical conditions including a form of minimality, \emph{given} any torsion-free connection on $\X$ with Christoffel symbols $\Gamma^a_{bc}$ (see Appendix B) the system is a pseudo-gradient system with respect to a pseudo-Riemannian metric $G$ if and only if along each solution of the system the input-output behavior of the \emph{variational system} (with inputs $\delta u(t)$ and outputs $\delta y(t)$) is equal to the input-output behavior of the \emph{dual variational system} given as (leaving out arguments $x(t), t$)
\bq
\label{duals}
\begin{array}{rcl}
\dot{p}_b & = & \left(\frac{\partial f_a}{\partial x_b} + 2 \Gamma^a_{bc} f_c\right)p_a + \sum_{j=1}^m u_j \left(\frac{\partial g_{ja}}{\partial x_b} + 2 \Gamma^a_{bc} g_{jc}\right)p_a +
\sum_{j=1}^m u_j^d\frac{\partial h_j}{\partial x_b} \\[3mm]
y_j^d & = & \sum_{i=1}^m p_i g_{ji} \quad j=1, \cdots,m,
\end{array}
\eq 
for $b=1, \cdots, n$, using Einstein's summation convention over index $a$. Here $u_j^d$ and $y_j^d$, $j=1,\cdots, m,$ are the dual variational inputs, respectively, dual variational outputs. Furthermore, there is exactly \emph{one} such pseudo-Riemannian metric $G$ whose Levi-Civita connection is equal to the given torsion-free connection. In fact, the pseudo-Riemannian metric $G(x)$ is the state space isomorphism between the variational state $\delta x$ and the dual variational state $p$.

With $\dot{x}_c$ denoting the $c$-th component of $\dot{x}$ in the original nonlinear system \eqref{systemaffine}, note that the first line of \eqref{duals} can be rewritten suggestively as
\bq
\dot{p}_b  =  \frac{\partial f_a}{\partial x_b} p_a + \sum_{j=1}^m u_j \frac{\partial g_{ja}}{\partial x_b}p_a + 2 \Gamma^a_{bc}  p_a \dot{x}_c
\eq
The precise implications of this theorem for the properties of the variational systems of a (Hessian) pseudo-gradient system are an important topic for further research. 
The same holds for the external characterization of \emph{Hessian} pseudo-gradient systems. (Note that the Christoffel symbols of the Levi-Civita connection of a Hessian pseudo-Riemannian metric $\nabla^2 K(x)$ take a special form; cf. Appendix B.)


%
%

\section{Combining nonlinear reciprocity with passivity}\label{sec:combination}
In this section we consider pseudo-gradient systems that are also passive, and aim at extending the linear theory as summarized in Sections \ref{subsec:recpas} and \ref{subsec:ph}. 

\subsection{Nonlinear passivity}
We start with considering passivity of nonlinear systems.
\begin{definition}
A nonlinear system $\Sigma$ given by \eqref{system} is passive if there exists $S: \X \to \mR$, which is bounded from below, such that
\bq
\label{passive}
\nabla S(x) F(x,u)
(x,u) \leq u^\top H(x,u)
\eq
for all $x \in \X, u \in \U$, where $\nabla S(x)$ is the row vector of partial derivatives of $S$. It is \emph{cyclo-}passive if $S$ is not necessarily bounded from below.
\end{definition}
From a geometric perspective the \emph{dissipation inequality} \eqref{passive} amounts to the following. 
Consider the system map $\Sigma_m$ (with an added minus sign, motivated by power flow considerations \cite{jeltsema}), taken together with the map $x \mapsto \nabla S(x)$, i.e., the map
\bq
(x,u) \mapsto (-F(x,u),H(x,u),\nabla S(x),u) \in T_x\X \times T_x^*\X \times \Y \times \U
\eq
Then the image of any $(x,u)$ should be an element of $T_x\X \times T_x^*\X \times \Y \times \U$ that is \emph{nonnegative} with respect to the duality pairing between $T_x\X$ and $T_x^*\X$, and the duality pairing between $\U$ and $\Y$. That is, for any $x,u$
\bq
<-F(x,u), \nabla S(x) > + <H(x,u), u >  \, \geq \, 0
\eq

The following observation will be used crucially in the rest of this section. Suppose $\X$ is an $n$-dimensional convex connected domain, and consider a cyclo-storage function $S: \X \to \mR$ with $\nabla^2 S(x)$ invertible. Denote $z:=\nabla S(x)$, and let $\Z$ be the co-domain of the mapping $x \mapsto \nabla S(x)$. Then consider the Legendre transform $S^*: \Z \to \mR$. It follows that (see Appendix A for more information) $x=\nabla S^*(z)$. Using $\dot{x} = \nabla^2 S^*(z) \dot{z}$, the system \eqref{system} is expressed in the new coordinates $z$ as
\bq
\begin{array}{rcl}
\nabla^2S^*(z) \dot{z} & = & F(\nabla S^*(z),u) \\[2mm]
y &= & H(\nabla S^*(z),u)
\end{array}
\eq
Then consider the function $\widetilde{S}(z):= S(\nabla S^*(z))$; i.e., the cyclo-storage function $S$ \emph{expressed in the $z$-coordinates}. It is in general \emph{differerent} from $S^*(z)$. In fact, equality holds if and only if $S$ is \emph{homogeneous of degree $2$} up to a constant; cf. Proposition \ref{A:prop}. 
The function $\widetilde{S}$ has the following properties.
\begin{proposition}
\label{tildeproperties}
Let $S: \X \to \mR$ be such that $\nabla^2 S(x)$ is invertible for all $x$, and let $S^*: \Z \to \mR$ be its Legendre transform. Then $\widetilde S(z):= S(\nabla S^*(z))$ satisfies
\begin{enumerate}
\item
$\widetilde{S}(z) = z^\top \nabla S^*(z) - S^*(z)$.
\item
$\nabla \widetilde{S}(z) = z^\top \nabla^2 S^*(z)$, in particular $\nabla \tilde{S}(0) = 0$.
\item
If $\nabla^2 S(x) \geq 0$ then $\widetilde{S}(z) \geq \widetilde{S}(0)$ for all $z \in \Z$.
\end{enumerate}
\end{proposition}
\begin{proof}
For the proof of the first two properties we refer to \cite{rockafellar,vds21}.  With regard to the last property, define the function $s(t,z):=\widetilde{S}(tz)$. Obviously, $s(1,z)=\widetilde{S}(z), s(0,z)=\widetilde{S}(0)$, while $\frac{\partial s}{\partial t}(t,z)= \nabla \widetilde{S}(tz) z$. Hence, by property (2)
\bq
\widetilde{S}(z) - \widetilde{S}(0)= \int_0^1 \frac{\partial s}{\partial t}(t,z) dt =  \int_0^1 \nabla \widetilde{S}(tz) z dt = \int_0^1 t z^\top \nabla^2 S^*(tz) z dt \geq 0
\eq
since $\nabla^2 S^*(z) \geq 0$.
\end{proof}
Compared with $S(x)$, the function $\widetilde{S}(z)$ has the special property that it has a minimum at $z=0$.
Furthermore, $\widetilde{S}(z)$ satisfies
\bq
\frac{d}{dt} \widetilde{S}(z) = \nabla S(\nabla S^*(z)) \nabla^2 S^*(z) \dot{z} = z^\top F(\nabla S^*(z),u) \leq u^\top H(\nabla S^*(z),u)
\eq
as follows by substituting $x= \nabla S^*(z)$ and $z= \nabla S(x)$ in \eqref{passive}. Thus if $S$ is a convex cyclo-storage function for \eqref{system}, then $\widetilde{S}$ satisfies the passivity inequality for the system expressed in $z$-coordinates, and in fact $\widetilde{S}$ is a \emph{storage function}. 

\subsection{Reciprocal port-Hamiltonian systems}
In order to motivate the approach presented in the next subsection on the combination of reciprocity with passivity in the nonlinear case, we will first consider reciprocity in the port-Hamiltonian case. See already \cite{vds11, jeltsema, vds19} for the main ideas exposed in this subsection.

Consider nonlinear port-Hamiltonian systems of the form
\bq
\label{isor}
\begin{array}{rcl}
\dot{z} & = & J(z) \nabla H(z) - R(\nabla H(z))+ g(z)u, \quad z \in \Z, u \in \mR^m \\[3mm]
y & = & g^\top (z) \nabla H(z), \quad y \in \mR^m
\end{array}
\eq
for some skew-symmetric matrix $J(z)$, where the mapping $R$ (corresponding to energy dissipation) satisfies $x^\top R(x) \geq 0$ for all $x$. The Hamiltonian $H$ represents the total storage energy of the system. For simplicity of exposition we will consider $\Z$ to be a coordinate neighborhood within $\mR^n$.
Throughout assume that the mapping from $z$ to $ x:= \nabla H(z)$ is {\it invertible} with invertible Hessian $\nabla^2 H(z)$. Then its inverse is given by $z = \nabla H^*(x)$, 
where $H^*(x)$ is the Legendre transform of $H$; cf. Appendix A. By substituting $z= \frac{\partial H^*}{\partial x}(x)$ into \eqref{isor} it follows that
\begin{equation}\label{eq:Brayton}
\frac{\partial^2 H^*}{\partial x^2}(x) \dot{x} = J(z)x -R(x) + g(z)u
\end{equation}
By finally substituting $z = \nabla H^*(x)$ one thus obtains a differential equation in the new state variables\footnote{The state variables $z$ are called the \emph{energy} variables in physical system modeling, since the Hamiltonian $H$ is expressed in them. The state variables $x=\nabla H(z)$ are called the \emph{co-energy} variables.} $x$. The system \eqref{eq:Brayton} turns out to be a passive Hessian pseudo-gradient system once we make a number of \emph{additional assumptions}.

First $(I)$, assume that there exist coordinates $z = (z_1, z_2)$ in which the matrices $J(z), g(z)$ are {\it constant}, and furthermore take the form
\begin{equation}\label{eq:aut1}
J = \begin{bmatrix} 0 & -P_c \\ P_c^\top & 0 \end{bmatrix}, \, \, g=\begin{bmatrix} g_1 \\ 0 \end{bmatrix}
\end{equation}
Second $(II)$, assume that in these coordinates the Hamiltonian $H$ \emph{splits} as $H(z_1,z_2) = H_1(z_1) + H_2(z_2)$, for certain functions $H_1$ and $H_2$. Writing accordingly $x=(x_1,x_2)$ with $x_1 =  \nabla H_1(z_1), x_2 =  \nabla H_2(z_2),$ it follows that the Legendre transform $H^*(x)$ splits as $H^*(x) =  H_1^*(x_1) + H_2^*(x_2)$. Then \eqref{eq:Brayton} takes the form
\begin{equation}\label{eq:Brayton1}
\begin{bmatrix} \nabla^2 H_1^*(x_1) & 0 \\ 0 & 
\nabla^2 H_2^*(x_2) \end{bmatrix}
\begin{bmatrix} \dot{x}_1 \\[2mm] \dot{x}_2 \end{bmatrix}
 = \begin{bmatrix} 0 & -P_c \\[2mm] P_c^\top & 0 \end{bmatrix} \begin{bmatrix} x_1 \\[2mm] x_2 \end{bmatrix} -
\begin{bmatrix} R_1(x) \\[2mm] R_2(x) \end{bmatrix} + \begin{bmatrix} g_1 \\[2mm] 0 \end{bmatrix}u
\end{equation}
Third $(III)$ assume that 
\begin{equation}
R_1(x) = \frac{\partial P_1}{\partial x_1}(x_1), \, \, R_2(x) = -\frac{\partial P_2}{\partial x_2}(x_2)
\end{equation}
for certain (Rayleigh dissipation) functions $P_1, P_2$. (This imposes an integrability condition on $R_1, R_2$.) Then, after multiplication of the equations in the last line in \eqref{eq:Brayton1} by $-1$, it follows that \eqref{eq:Brayton1} can be rewritten as
\begin{equation}\label{eq:Brayton2}
\begin{array}{rcl}
 \begin{bmatrix} \nabla^2 H_1^*(x_1) & 0 \\[2mm] 0 & 
-\nabla^2 H_2^*(x_2) \end{bmatrix}
\begin{bmatrix} \dot{x}_1 \\[2mm] \dot{x}_2 \end{bmatrix}& = &
- \begin{bmatrix} \frac{\partial P}{\partial x_1} \\[2mm] \frac{\partial P}{\partial 
x_2} \end{bmatrix} + \begin{bmatrix} g_1 \\[2mm] 0 \end{bmatrix}u \\[6mm]
y & = & g_1^\top x_1
\end{array}
\end{equation}
where $P$ is the {\it mixed-potential function} defined as
\begin{equation}
P(x_1, x_2) := P_1(x_1) + P_2(x_2) + x_1^TP_c x_2
\end{equation}
The equations \eqref{eq:Brayton2} define a Hessian pseudo-gradient system with respect to the Hessian pseudo-Riemannian metric $\nabla^2 K(x)$ with $K(x_1,x_2) = H_1^*(x_1) - H_2^*(x_2)$ and internl potential function $P$. By definition, the port-Hamiltonian system \eqref{isor} is cyclo-passive, and \emph{passive} if the Hamiltonian $H(z_1,z_2)= H_1(z_1) + H_2(z_2)$ is bounded from below. 

Thus if fourth $(IV)$ the functions $H_1, H_2$ are assumed to be bounded from below, then $\widetilde{H}$ defined as
\bq
\widetilde{H}(x_1,x_2) := H_1(\nabla H_1^*(x_1)) + H_2(\nabla H_2^*(x_2))
\eq
is a storage function for \eqref{eq:Brayton2}, and the system \eqref{eq:Brayton2} is a \emph{passive Hessian pseudo-gradient system}.

\subsection{Passive Hessian pseudo-gradient systems}
\label{subsec:pashes}
Combining general nonlinear reciprocity with passivity, fully extending the linear theory of \cite{willems72b} as discussed in Section \ref{subsec:recpas}, seems complicated. Therefore we will restrict ourselves to \emph{Hessian} pseudo-gradient systems. 

One of the problems in extending the linear theory, even in the Hessian case, is that the \emph{compatibility condition} $Q=GQ^{-1}G$ (see \eqref{complinear}) admits many, generally not equivalent, nonlinear generalizations. Indeed, in the linear case, $K(x)=\frac{1}{2}x^\top G x$ for some symmetric matrix $G$, and the storage function is $S(x) = \frac{1}{2}x^\top Q x$ for some symmetric $Q>0$. Hence the Legendre transform $S^*(z)$ is given as $\frac{1}{2}z^\top Q^{-1}z$, and thus the compatibility condition $Q=GQ^{-1}G$ can be formulated as the equality $S(x)=S^*(\nabla K(x))$; suggesting a nonlinear generalization. On the other hand, the compatibility condition \eqref{complinear} is also equivalent to $G=QG^{-1}Q$, which instead would suggest the (different!) nonlinear generalization $K(x)= K^*(\nabla S(x))$. 

On the other hand, in the previous subsection it was discussed how under the four assumptions $(I) - (IV)$ port-Hamiltonian systems can be formulated as passive Hessian pseudo-gradient systems. This will serve as the starting point for the approach taken in the current subsection.%

Consider a Hessian pseudo-gradient system \eqref{gradient}, with respect to a Hessian pseudo-Riemannian metric $G(x)=\nabla^2 K(x)$ and $\sigma=I$. Furthermore, consider, as before, a potential function $V(x,u)$ of the form $V(x,u)= P(x) - x^\top gu$ for some internal potential function $P(x)$ and constant input matrix $g$. Thus the Hessian pseudo-gradient system is given as
\bq
\label{gradienthessian1}
\begin{array}{rcl}
\nabla^2 K(x) \dot{x} & = &  - \frac{\partial P}{\partial x}(x) + gu \\[2mm]
y & = &  g^\top x
\end{array}
\eq
Suppose now that \eqref{gradienthessian1} is \emph{passive}, and that there exists a storage function of the form $S^*(\nabla S(x))$, for some $S:\X \to \mR$ with invertible Hessian matrix and Legendre transform $S^*$, and coordinates $x=(x_1,x_2)$ for $\X$ such that
\bq
S(x_1,x_2) = S_1(x_1) + S_2(x_2), \quad K(x_1,x_2) = S_1(x_1) - S_2(x_2)
\eq
This means that the Hessian pseudo-gradient system is given as
\bq
\begin{array}{rcl}
\bma \nabla^2 S_1(x_1) & 0 \\[2mm] 0 & -\nabla^2 S_2(x_2) \ema \bma \dot{x}_1 \\[2mm] \dot{x}_2 \ema & = &
- \bma \frac{\partial P}{\partial x_1}(x_1,x_2) \\[2mm] \frac{\partial P}{\partial x_1}(x_1,x_2)  \ema + \bma g_1 \\[2mm] g_2 \ema u \\[5mm]
y & = & \bma g^\top_1 & g^\top_2 \ema x
\end{array}
\eq
Furthermore, since $S^*(\nabla S(x))$ is assumed to be a storage function,
\bq
\begin{array}{l}
\frac{d}{dt}[S_1^*(\nabla S_1(x_1)) + S_2^*(\nabla S_2(x_2))] = x_1^\top \nabla^2 S_1 \dot{x}_1 + x_2^\top \nabla^2 S_2 \dot{x}_2 =\\[2mm]
-x_1^\top \frac{\partial P}{\partial x_1}(x_1,x_2) + x_2^\top \frac{\partial P}{\partial x_2}(x_1,x_2) + u^\top \bma g^\top_1 & -g^\top_2 \ema x \leq u^\top \bma g^\top_1 & 
g^\top_2 \ema x
\end{array}
\eq
for all $x_1,x_2,u$. This implies (substituting successively $x_1=0, x_2=0,u=0$)
\bq
g_2=0, \qquad x_1^\top \frac{\partial P}{\partial x_1}(x_1,0) \geq 0, \; \; x_2^\top \frac{\partial P}{\partial x_2}(0, x_2) \leq 0
\eq
This can be regarded as a nonlinear generalization of \eqref{portham1}.
%


\subsection{Nonlinear relaxation systems}
\label{subsec:relaxation}
A special class of nonlinear passive Hessian pseudo-gradient systems is the class of \emph{nonlinear relaxation systems}; cf. Section \ref{subsec:recpas} for the linear case. These systems will be defined as follows.
\begin{definition}
A nonlinear relaxation system is a gradient system \eqref{gradient} with Hessian Riemannian metric $G(x)=\nabla^2 K(x) >0$ and\footnote{For other choices of $\sigma$ one needs to adapt the condition \eqref{relpot}. See Example \ref{nonRC}.} $\sigma=I$, where the potential function $V$ is assumed to satisfy for all $x,u$
\bq
\label{relpot}
x^\top \frac{\partial V}{\partial x}(x,u) - u^\top \frac{\partial V}{\partial u}(x,u) \geq 0
\eq
\end{definition}
Nonlinear relaxation systems are much more amenable for analysis than general nonlinear (Hessian) pseudo-gradient systems.
Importantly, nonlinear relaxation systems are \emph{passive}, as stated in the next proposition.
\begin{proposition}
A nonlinear relaxation system is \emph{passive}, with storage function $S(x):= K^*(\nabla K(x))$. 
\end{proposition}
\begin{proof}
By direct computation
\bq
\begin{array}{rcl}
\frac{d}{dt} S(x) &=  &\nabla K^*(\nabla K(x)) \nabla^2 K(x) \dot{x} = x^\top \nabla^2 K(x) \dot{x} \\[2mm]
 &= & -x^\top \frac{\partial V}{\partial x}(x,u) \leq - u^\top \frac{\partial V}{\partial u}(x,u) = u^\top y,
\end{array}
\eq 
where the inequality follows from \eqref{relpot}. Furthermore, Proposition \ref{tildeproperties} (with the notation for $x$- and $z$-variables swapped) implies that $S(x) \geq S(0)$ for all $x \in \X$, and thus $S$ is bounded from below and hence a true storage function.
\end{proof}
Note that, as mentioned before, in general $S(x)=K^*(\nabla K(x))$ is different from $K$, with equality if and only if $K$ is, up to a constant, homogeneous of degree $2$ (e.g., $K$ is quadratic); cf. Proposition \ref{A:prop}. Also note that the condition \eqref{relpot} can be interpreted as a weak convexity condition on $V$.

In the case that $V(x,u)= P(x) - \sum_{j=1}^m C_j(x)u_j$, condition \eqref{relpot} amounts to the inequality
\bq
x^\top \frac{\partial P}{\partial x}(x) - \sum_{j=1}^m x^\top \frac{\partial C_j}{\partial x}(x)u_j + u^\top \sum_{j=1}^m C_j(x) \geq 0
\eq
for all $x,u$, and thus to 
\bq
x^\top \frac{\partial P}{\partial x}(x) \geq 0, \quad C_j(x)= x^\top \frac{\partial C_j}{\partial x}(x), \quad j=1, \cdots,m
\eq
(i.e., $C_j(x)$ homogeneous of degree $1$; see also Appendix A).

\section{Examples}\label{sec:examples}
\begin{example}[Brayton-Moser formulation of nonlinear RLC-circuits, \cite{BM1,BM2}]
\label{ex:BM}
{\rm The Brayton-Moser formulation of RLC electrical networks with linear capacitors and inductors, and nonlinear resistors/conductors is given as the Hessian pseudo-gradient system
\begin{equation}\label{BM1}
\begin{bmatrix} L & 0 \\ 0 & -C \end{bmatrix} \dot{x}  =  - \frac{\partial P}{\partial x}(x) \, , \quad x = \begin{bmatrix} I \\ V \end{bmatrix},
\end{equation}
with $L$ a diagonal matrix of inductances, $C$ a diagonal matrix of capacitances, and $P(x) = P_1(I) + P_2(V) + I^\top \Lambda V$ the mixed-potential function. Here $P_1$ is the content function of the nonlinear resistors (parametrized by currents $I$), $P_2$ the co-content function of the conductors (parametrized by voltages $V$), while $I^\top \Lambda V$ is a coupling term for a certain matrix $\Lambda$ reflecting the topology of the network. Note that the Hessian pseudo-Riemannian metric (constant in this case) is determined by 
\bq
K(I,V)= \frac{1}{2}L I^2 - \frac{1}{2}C V^2
\eq
A port-Hamiltonian formulation is obtained by the transformation of \eqref{BM1} into the state variables $z=(LI,CV)=:(\varphi,Q)$ (flux linkages and charges), leading to the Hamiltonian $\frac{1}{2}\varphi^\top L^{-1} \varphi + \frac{1}{2}Q C^{-1}Q$ (magnetic plus electric energy).
}
\end{example}

\begin{example} [Swing equation model of a power network, \cite{vdsstegink,vds19}]\label{swing}
{\rm The swing equation model of a power network with incidence matrix $D$ is given by the port-Hamiltonian system
\bq
\label{grad}
\begin{array}{rcl}
\begin{bmatrix} \dot{p} \\[2mm] \dot{q} \end{bmatrix} & = &
\begin{bmatrix} -A & D\\[2mm] D^\top & 0 \end{bmatrix} 
\begin{bmatrix} \frac{\partial H}{\partial p}(p,q) \\[2mm]  \frac{\partial H}{\partial q}(p,q) \end{bmatrix} + 
\begin{bmatrix} I \\[2mm] 0 \end{bmatrix} u \\[8mm]
y & = & \frac{\partial H}{\partial p}(p,q) ,
\end{array}
\eq
where $p$ is the vector of momenta corresponding to the synchronous machines at the nodes of the power network, and $q$ is the vector the angle differences across each transmission line (edge of the network). The Hamiltonian is given as
\bq
H(p,q) =\frac{1}{2} p^\top M^{-1}p - \sum_{j=1}^k \gamma_j \cos q_j,
\eq
for a positive diagonal mass matrix $M$, and constants $\gamma_j$ determined by the physical properties of the $j$-th transmission line and the voltages at its adjacent nodes (which in the swing equation model are assumed to be constant). The output $y$ is given by $\frac{\partial H}{\partial p}(p,q)=M^{-1}p=:\omega$ representing the frequency deviations (with respect to a nominal value, e.g., $50$Hz) at each node. This defines a Hessian pseudo-gradient system.
In fact, the $x$-variables are given by
\bq
\begin{array}{rcll}
\omega & := & \frac{\partial H}{\partial p}(p)= M^{-1}p  \quad & \mbox{ (frequency deviations at the nodes) }\\[2mm]
\pi & := & \frac{\partial H}{\partial q}(q)= \Gamma \Sin q , \quad & \mbox{ (power flows through the lines) }
\end{array}
\eq
where $\Sin$ denotes the vector sinus function $\Sin q =(\sin q_1, \cdots, \sin q_k)$, with $k$ the number of edges, and $\Gamma$ the diagonal matrix with diagonal elements $\gamma_1, \cdots, \gamma_k$.
Then \eqref{grad} can be rewritten as the Hessian pseudo-gradient system
\bq
\begin{bmatrix} M & 0 \\[2mm] 0 & -L(\pi)
\end{bmatrix}
\begin{bmatrix} \dot{\omega} \\[2mm]  \dot{\pi}\end{bmatrix} = 
- \begin{bmatrix} \frac{\partial P}{\partial \omega}(\omega, \pi) \\[2mm]  \frac{\partial P}{\partial \pi}(\omega, \pi)\end{bmatrix} +
\begin{bmatrix} I \\[2mm]0 \end{bmatrix}u, \quad y= \omega,
\eq
where $L(\pi)$ is the positive diagonal matrix with $k$-th diagonal element
$\frac{1}{\sqrt{\gamma_k^2 - \pi_k^2}}$.
In fact, $H_1(p)=\frac{1}{2} p^T M^{-1}p$ and $H_2(q)=- \sum_{j=1}^k \gamma_j \cos q_j$, with Legendre transforms
\bq
H_1^*(\omega) = \frac{1}{2} \omega^\top M \omega, \quad H_2^*(\pi)= \sum_{j=1}^k \pi_j \arcsin \frac{\pi_j}{\gamma_j} + \gamma_j \cos (\arcsin \frac{\pi_j}{\gamma_j}),
\eq
yielding $K(\omega,\pi)= \frac{1}{2} \omega^\top M \omega - \sum_{j=1}^k \pi_j \arcsin \frac{\pi_j}{\gamma_j} + \gamma_j \cos (\arcsin \frac{\pi_j}{\gamma_j})$.
The mixed-potential function is
\bq
P(\omega,\pi)= \pi^\top D\omega + \frac{1}{2}\omega^\top A \omega
\eq
Following the theory of Section \ref{subsec:pashes} the system is passive with storage function
\bq
\frac{1}{2}\omega^\top M \omega - \sum_{j=1}^k \gamma_j \cos \frac{\pi_j}{\gamma_j}
\eq
(i.e., the original Hamiltonian $H$ expressed in the co-energy variables $\omega, \pi$).
}
\end{example}
\begin{example}
\label{nonRC}
{\rm Consider a nonlinear RC electrical circuit, with nonlinear \emph{conductors} at the \emph{edges} and grounded nonlinear \emph{capacitors} at part of the \emph{nodes} ($c$), while the remaining nodes are \emph{terminals} ($t$); see also \cite{cams22}.
Decompose the incidence matrix as $D = \bma D_c \\ D_t \ema$. The nonlinear conductors at the $k$ edges are given by characteristics $I_j = G_j(V_j)$, $j=1, \cdots, k$, where the functions $G_j$ are assumed to be \emph{monotone} and continuous, implying the existence of convex functions $\widehat{W}_j(V_j)$ with $G_j(V_j)= \frac{d\widehat{W}_j}{dV_j}(V_j), j=1, \cdots,m$. Furthermore assume $G_j(0)=0$. Define the convex function $\widehat{W}(V_1, \ldots, V_m) := \sum_{j=1}^k \widehat{W}_j(V_j)$. It follows that the vector of currents $I$ through the edges is given as $I= \frac{\partial \widehat{W}}{\partial V}(V)$, where $V$ is the vector of voltages across the edges.

By Kirchhoff's current law the vector $J$ of currents incoming at the nodes relates to the currents $I$ through the edges as $\bma J_c \\ J_t \ema = DI$. Furthermore, by Kirchhoff's voltage law $V=D^T \psi$, where $\psi = \bma \psi_c \\ \psi_t \ema$ is the vector of nodal voltage potentials. Define finally the convex function $W(\psi) := \widehat{W}(D^T \psi)$.
Then 
\bq
\frac{\partial W}{\partial \psi}(\psi) = D \frac{\partial \widehat{W}}{\partial V}(D^T \psi) = D I= \bma J_c \\ J_t \ema,
\eq
while $\frac{\partial W}{\partial \psi}(0)=0$. 
Furthermore, the grounded nonlinear capacitors with charges $Q$ satisfy $\dot{Q}=-J_c$ and $\psi_c = \nabla H(Q)$, with $H(Q)$ the electric energy stored at the capacitors with vector of charges $Q$. Hence the dynamics is given as
\bq
\begin{array}{rcl}
\dot{Q} & = & - \frac{\partial W}{\partial \psi_c}(\nabla H(Q), \psi_t) \\[4mm]
J_t & = & \frac{\partial W}{\partial \psi_t}(\nabla H(Q), \psi_t) 
\end{array}
\eq
with inputs $\psi_t$ (voltage potentials at terminals) and outputs $J_t$ (incoming currents at the terminals). Since $W$ is convex in $\psi$, this is a maximal cyclically monotone port-Hamiltonian system with respect to the signature matrix $\sigma= -I$; cf. Proposition \ref{prop:monotone}. 

Finally, assume that the electric energy $H(Q)$ satisfies $\nabla^2 H(Q)>0$. Then in the new state coordinates $\psi_c=\nabla H(Q)$, the equations take the equivalent form
\bq
\begin{array}{rcl}
\nabla^2 H^*(\psi_c) \dot{\psi}_c & = & - \frac{\partial W}{\partial \psi_c}(\psi_c, \psi_t) \\[4mm]
J_t & = & \frac{\partial W}{\partial \psi_t}(\psi_c, \psi_t) ,
\end{array}
\eq
which can be recognized as a \emph{relaxation system} (however now with $\sigma=-I$). Following the set-up in Section \ref{subsec:relaxation} the storage function is given as $S(\psi_c):= H(\nabla H^*(\psi_c))$. (Note that the function $K$ determining the Hessian Riemannian metric is given as $K(\psi_c) = H^*(\psi_c)$.) Indeed, since $\frac{\partial W}{\partial \psi}(0)=0$ and $W$ is convex
\bq
\psi_c^\top \frac{\partial W}{\partial \psi_c}(\psi)   + \psi_t^\top \frac{\partial W}{\partial \psi_t}(\psi) \geq 0
\eq
Hence
\bq
\frac{d}{dt} S(\psi_c) = \nabla H (\nabla H^*(\psi_c)) \nabla^2 H^*(\psi_c) = - \psi_c^\top \frac{\partial W}{\partial \psi_c}(\psi) \leq J_t^\top \psi_t
\eq
}
\end{example}
%

\section{Conclusions}
\label{sec:conclusions}
A number of steps have been taken to generalize the results of \cite{willems72b} to the nonlinear case. First of all, a clear geometric definition of reciprocity of nonlinear input-state-output systems has been given extending the geometric formulation of reciprocity in the linear case. Furthermore the notion of a pseudo-gradient system has been specialized to Hessian pseudo-Riemannian systems, which is a physically well-motivated class of pseudo-gradient systems that is more amenable for analysis than general pseudo-gradient systems. In particular, the combination with passivity is more easy. The ensuing definition of a relaxation system with an explicit storage function should be a good starting point for analyzing their input-output properties. Furthermore, the results obtained in this paper demonstrate the importance of Legendre transformations and convex analysis. Indeed, although the examples given in this paper are all of a physical nature, this can be extended to examples originating from convex optimization algorithms (in continuous time); see already \cite{cams22}.

In general, the problem of characterization of input-output properties of (Hessian) pseudo-gradient systems is still open (see \cite{chaffey,willems72b,willems76} for the linear case), although the results of \cite{cortes} should be helpful. Furthermore, only partial results concerning the combination of reciprocity and passivity have been obtained.

\section{Appendix A: Legendre transformation and its properties}

Consider a differentiable function $K: \X \to \mR$ on an $n$-dimensional connected convex domain $\X$, with invertible Hessian matrix, i.e., $ \det \nabla^2 K(x) \neq 0$ for all $x \in \X$. This implies that the map $x \mapsto z:= \nabla K(x)$ is injective. Denote the co-domain of this map by $\Z$. Then the Legendre transform of $K$ is the function $K^*: \Z \to \mR$ defined as
\bq
K^*(z):=  z^\top x - K(x), \quad z =\nabla K(x),
\eq
where $x$ is solved from $z = \nabla K(x)$. 

In case $\nabla^2 K(x) \geq 0$ for all $x \in \X$, then $K$ is \emph{convex} (and \emph{strictly convex} if $\nabla^2 K(x) > 0$), and the Legendre transform of $K$ is given as
\bq
K^*(z)=  \sup_x z^\top x - K(x)
\eq
Furthermore, $\Z$ is convex and $K^*: \Z \to \mR$ is a convex function. $K^*$ in this case is also called the \emph{convex conjugate} of $K$.

The following properties of the Legendre transform are well-known; see for example \cite{rockafellar}
\bq
\begin{array}{l}
\left( K^*\right)^* = K \\[2mm]
\nabla K(\nabla K^*(z)) = z, \quad \nabla K^*(\nabla K(x)) = x \\[2mm]
\nabla^2 K^*(z) = \left(\nabla^2 K(\nabla K^*(z))\right)^{-1}, \quad \nabla^2 K(x) = \left(\nabla^2 K^*(\nabla K(x))\right)^{-1}
\end{array}
\eq
Two other useful properties, perhaps less well-known, are the following.
\begin{proposition}\label{A:prop}
Let $K: \X \to \mR$ and $K^*: \Z \to \mR$ its Legendre transform. Consider the function $L(x):= K^*(\nabla K(x))$. 
Then 
\bq
\nabla L(x) = \nabla K^*(\nabla K(x)) \nabla^2 K(x) = x^\top \nabla^2 K(x)
\eq
(In particular, $\nabla L(0)=0$.)

Furthermore, $K^*(\nabla K(x))=K(x)$ if and only if $K(x) -K(0)$ is homogeneous of degree $2$.
\end{proposition}
\begin{proof}
The first statement follows from a direct computation and the fact that $\nabla K^*(\nabla K(x))= x$.
%

For the second statement, let $K(0)=0$, and assume $K$ is homogeneous of degree $2$. Then consider $K^*(\nabla K(x))= \nabla K(x) x - K(x)$. Since $K$ is homogeneous of degree $2$, by Euler's homogenous function theorem $\nabla K(x) x = 2K(x)$. Substitution yields $K^*(\nabla K(x))=K(x)$.

Conversely, if $K^*(\nabla K(x))=K(x)$ then $\nabla K(x) = x^\top \nabla^2 K(x)$. This means that
\bq
\label{Euler2}
\sum_{j=1}^n x_j \frac{\partial^2 K}{\partial x_j \partial x_i}(x) = \frac{\partial K}{\partial x_i}(x)
\eq
Define the functions $k_i(x)= \frac{\partial K}{\partial x_i}(x), i=1, \cdots,n$. Then \eqref{Euler2} is the same as $\sum_{j=1}^n x_j \frac{\partial k_i}{\partial x_j}(x) = k_i(x)$, $ i=1, \cdots,n$. Thus by Euler's homogeneous function theorem the functions $k_i$ are all homogeneous of degree $1$. 
Finally define the function $h(t,x):=K(tx)$. Obviously, $h(1,x)=K(x), h(0,x)=K(0)$, and furthermore $\frac{\partial h}{\partial t}(t,x)= \sum_{i=1}^n k_i(tx)x_i$. Hence
\bq
K(x) - K(0)= \int_0^1 \frac{\partial h}{\partial t}(t,x) dt = \sum_{i=1}^n \int_0^1 k_i(tx)x_i dt = \sum_{i=1}^n  k_i(x)x_i \int_0^1 tdt
\eq
where the last equality follows from the fact that $k_i$ are all homogeneous of degree $1$, and thus $k_i(tx)=tk_i(x)$. Therefore, $K(x) - K(0)= \frac{1}{2}\sum_{i=1}^n  \frac{\partial K}{\partial x_i}(x)x_i$, showing, again by Euler's theorem, that $K(x) -K(0)$ is homogeneous of degree $2$.
\end{proof}
\begin{remark}
Obviously, if $K$ is the quadratic function $K(x)=\frac{1}{2}x^\top G x$, then it is homogeneous of degree $2$ and therefore $K^*(\nabla K(x))=K(x)$. Indeed, its Legendre transform is $K^*(z)=\frac{1}{2}z^\top G^{-1} z$, and thus, substituting $z=\nabla K(x)=Gx$, $K^*(\nabla K(x))= \frac{1}{2}(Gx)^\top G^{-1} Gx = K(x)$. 
\end{remark}

\section{Appendix B: Connections of (Hessian) pseudo-Riemannian metrics}
A connection $D$ on an $n$-dimensional manifold $\X$ is an assignment $(X,Y)  \mapsto  D_XY$ from any two vector fields $X,Y$ on $\X$ to a new vector field $D_XY$ on $\X$ (the covariant derivative of $Y$ along $X$), which is $\mR$-bilinear, and satisfies $D_{fX} Y = f D_XY$ and $D_{X} fY = f D_XY + X(f)Y$ for any vector fields $X,Y$ on $\X$ and functions $f:\X \to \mR$. In local coordinates $x_1,\cdots,x_n$, a connection $D$ is determined by its \emph{Christoffel symbols} $\Gamma^k_{ij}(x), i,j,k=1, \cdots,n,$ as
\bq
D_{\frac{\partial}{\partial x_i}} {\frac{\partial}{\partial x_j}} = \sum_{k=1}^n\Gamma^k_{ij}(x) {\frac{\partial}{\partial x_k}}
\eq
The connection is called \emph{torsion-free} if the Christoffel symbols $\Gamma^k_{ij}(x)$ are symmetric in $i,j$.

Any pseudo-Riemannian metric $G$ on $\X$ determines a unique torsion-free \emph{connection} (called the Levi-Civita connection in case $G$ is a true Riemannian metric). The Christoffel symbols of the Levi-Civita connection of a pseudo-Riemannian metric $G$ are defined as $\Gamma^k_{ij}(x) := \sum_{\ell} G^{k\ell}(x) \Gamma_{\ell ij}(x)$, where $G^{k\ell}(x)$ denotes the $(k,\ell)$-th element of the inverse matrix $G(x)^{-1}$ and
\bq
\Gamma_{\ell ij}= \frac{1}{2}\left( \frac{\partial G_{\ell i}(x)}{\partial x_j} + \frac{\partial G_{\ell j}(x)}{\partial x_i} - \frac{\partial G_{ij}(x)}{\partial x_\ell} \right)
\eq 
In case of a \emph{Hessian} pseudo-Riemannian metric $G(x)=\nabla^2 K(x)$ the Christoffel symbols of the Levi-Civita connection simplify to
\bq
\Gamma_{kij}(x)= \frac{1}{2} \frac{\partial^3 K(x)}{\partial x_k \partial x_i \partial x_j}, \quad \Gamma^k_{ij}(x)= \frac{1}{2} \sum_{\ell=1}^n \frac{\partial^2 K^*(\nabla K(x))}{\partial z_k \partial z_\ell} \frac{\partial^3 K(x)}{\partial x_\ell \partial x_i \partial x_j},
\eq
where we have used $\left( \nabla^2 K(x)\right)^{-1} = \nabla^2 K^*(\nabla K(x))$, with $K^*: \Z \to \mR$ the Legendre transform of $K$; cf. Appendix A. Furthermore, by invertibility of $\nabla^2 K^*(z)$ it follows that the Christoffel symbols $\Gamma^k_{ij}(x)$ of the Levi-Civita connection of $\nabla^2 K(x)$ are zero if and only if $\frac{\partial^3 K(x)}{\partial x_k \partial x_i \partial x_j}=0$; that is, if and only if $K$ is a quadratic-affine function (and thus $\nabla^2 K(x)$ is constant).


\begin{thebibliography}{00}
%

\bibitem{arnold}
V.I. Arnold,
\newblock {\em Mathematical Methods of Classical Mechanics},
\newblock Springer, 2nd edition, 1989.

\bibitem{BM1}
R.K. Brayton, J.K. Moser, A theory of nonlinear networks I, {\it Quart. Appl. Math.}, 22(1):1--33, 1964.

\bibitem{BM2}
R.K. Brayton, J.K. Moser, A theory of nonlinear networks II, {\it Quart. Appl. Math.}, 22(2): 81--104, 1964.
%


\bibitem{goncalves}
J.A. Basto Goncalves.
Nonlinear controllability and observability with applications to gradient systems.
Ph.D. Thesis, Univ. of Warwick, 1981.

\bibitem{cam14}
M.K. Camlibel, L. Ianneli and F. Vasca.
Passivity and complementarity.
{Math. Program., Ser. A}, 145:531--563, 2014.

\bibitem{cams22}
M.~K. Camlibel and A.~J. {van der Schaft}.
\newblock Port-Hamiltonian systems theory and monotonicity.
\newblock {\em SIAM Journal on Control and Optimization}, 61(4):2193--2221, 2023.

\bibitem{cams13}
M.~K. Camlibel and A.~J. {van der Schaft}.
\newblock Incrementally port-{H}amiltonian systems.
\newblock In {\em 52nd IEEE Conference on Decision and Control}, pp. 2538--2543. IEEE, 2013.

\bibitem{chaffey}
T. Chaffey, H. J. van Waarde, R. Sepulchre.
Relaxation systems and cyclic monotonicity.
\newblock In {\em 62nd IEEE Conference on Decision and Control}, IEEE, 2023.
Extended version: arXiv:2312.03389, 2023.

\bibitem{cortes}
J. Cortes, A.J. van der Schaft, P.E. Crouch.
Characterization of gradient control systems.
{\it SIAM J. Control Optim.}, 44(4):1192--1214, 2005.
  
\bibitem{crouch}
P.E. Crouch: Geometric structures in systems theory.
{\it Proc. IEE. D. Control Theory and Applications},128(5):242--252, 1981.

\bibitem{crouchvds}
P.E. Crouch, A.J. van der Schaft.
{\it Variational and Hamiltonian control systems}.
Lect. Notes in Control and Information Sciences, Vol. 101,
Springer-Verlag, Berlin, 1987.


\bibitem{duistermaat}
J.J. Duistermaat.
On Hessian Riemannian structures.
{\it Asian J. Math.}, 5(1): 79--92, 2001.


\bibitem{pates19}
R. Pates, C. Bergeling, and A. Rantzer. 
On the optimal control of relaxation systems.
In {\em 58th Conference on Decision and Control (CDC)}, pp. 6068--6073. IEEE, 2019. 

\bibitem{pates22}
R. Pates, Passive and reciprocal networks: from simple models to simple optimal controllers.
{\it IEEE Control Systems Magazine}, 42(3):73--92, 2022.

\bibitem{rockafellar}
R. T. Rockafellar, {\it Convex Analysis}. Princeton University Press, 1970.

\bibitem{shim}
H. Shima, K. Yagi.
Geometry of Hessian manifolds.
{\it Differential Geometry and its Applications}, 7:277--290, 1997.

\bibitem{vds83}
A.J. van der Schaft.
{\it System Theoretic Descriptions of Physical Systems}.
CWI Tract no.3, CWI, Amsterdam, 1983.

\bibitem{vds11}
A.J. van der Schaft.
On the relation between port-Hamiltonian and gradient systems.
18th IFAC World Congress, Milano (Italy).
pp. 3321--3326, 2011.


\bibitem{vds19}
A.J. van der Schaft,
Port-Hamiltonian modeling for control,
{\it Annual Review of Control, Robotics, and Autonomous Systems}, 3, 2019.


\bibitem{jeltsema}
A.~{van der S}chaft and D.~Jeltsema.
\newblock Port-{Hamiltonian} systems theory: {A}n introductory overview.
\newblock {\em Foundations and Trends in Systems and Control}, 1(2-3):173--378, 2014.


\bibitem{vds21}
A.J. van~der Schaft and B. Maschke.
Dirac and Lagrange algebraic constraints in nonlinear port-Hamiltonian systems,		
{\it Vietnam Journal of Mathematics}, 48(4):929--939, 2020.


\bibitem{vds23}
A.J. van~der Schaft and V.~Mehrmann,
Linear port-Hamiltonian DAE systems revisited.
{\it Systems \& Control Letters}, 177, 105564, 2023.


\bibitem{vdsstegink}
A.J. van der Schaft, T.W. Stegink.
Perspectives in modelling for control of power networks.
{\it Annual Reviews in Control}, 41:119--132, 2016.

%

\bibitem{willems72b}
J. C. Willems.
Dissipative dynamical systems, Part II:  Linear systems with quadratic supply rates.
{\it Archive for Rational Mechanics and Analysis}, 45: 352--393, 1972.

\bibitem{willems76}
J. C. Willems, Realization of systems with internal passivity and symmetry constraints.
{\it Journal of the Franklin Institute}, 301(6):605--621, 1976.
%

\end{thebibliography}
\end{document}